\documentclass{amsart}
\usepackage{amssymb}
\usepackage{amsfonts}
\usepackage{amssymb}
\usepackage{amsmath}
\usepackage{amsthm}
\usepackage{enumerate}
\usepackage{tabularx}
\usepackage{centernot}
\usepackage{mathtools}
\usepackage{stmaryrd}
\usepackage{amsthm,amssymb}
\usepackage{etoolbox}
\usepackage{tikz}
\usepackage{amssymb}
\usetikzlibrary{matrix}
\usepackage{tikz-cd}
\usepackage{tikz}
\usepackage{marginnote}
\definecolor{mygray}{gray}{0.85}
\usepackage[linecolor=black, backgroundcolor=mygray,colorinlistoftodos,prependcaption,textsize=small]{todonotes}
\usepackage{xargs}  

\usepackage{hyperref}

\renewcommand{\leq}{\leqslant}
\renewcommand{\geq}{\geqslant}

\renewcommand{\trianglelefteq}{\trianglelefteqslant}
\newcommand{\mrm}[1]{\mathrm{#1}}

\theoremstyle{plain}

\usepackage[backgroundcolor=mygray,colorinlistoftodos,prependcaption,textsize=small]{todonotes}
\usepackage{xcolor}

\makeatletter
\def\subsection{\@startsection{subsection}{3}%
  \z@{.5\linespacing\@plus.7\linespacing}{.3\linespacing}%
  {\bfseries\centering}}
\makeatother

\makeatletter
\def\subsubsection{\@startsection{subsubsection}{3}%
  \z@{.5\linespacing\@plus.7\linespacing}{.3\linespacing}%
  {\centering}}
\makeatother

\makeatletter
\def\myfnt{\ifx\protect\@typeset@protect\expandafter\footnote\else\expandafter\@gobble\fi}
\makeatother

\newtheorem{theorem}{Theorem}[section]

\newtheorem{corollary}[theorem]{Corollary}
\newtheorem{definition}[theorem]{Definition}
\newtheorem{lemma}[theorem]{Lemma}
\newtheorem{proposition}[theorem]{Proposition}

\newtheorem{hypothesis}[theorem]{Hypothesis}

\newtheorem{remark}[theorem]{Remark}
\newtheorem{notation}[theorem]{Notation}

\newcounter{claimcounter}

\makeatletter
\@namedef{subjclassname@2020}{%
  \textup{2020} Mathematics Subject Classification}
\makeatother

\begin{document}

\begin{abstract} In \cite{coarse} it was asked if equality on the reals is sharp as a lower bound for the complexity of topological isomorphism between oligomorphic groups.
We prove that under the assumption of weak elimination of imaginaries this is indeed the case. Our methods are model theoretic and they also have applications on the classical problem of reconstruction of isomorphisms of permutation groups from (topological) isomorphisms of automorphisms groups. As a concrete application, we give an explicit description of $\mrm{Aut}(\mrm{GL}(V))$ for any vector space $V$ of dimension $\aleph_0$ over a finite field, in affinity with the  classical description for finite dimensional \mbox{spaces due to Schreier and van der Waerden.}
\end{abstract}

\title[The Isomorphism Problem for Oligomorphic Groups with WEI]{The Isomorphism Problem for Oligomorphic Groups with Weak Elimination of Imaginaries}


\thanks{The author was partially supported by project PRIN 2022 ``Models, sets and classifications", prot. 2022TECZJA. The author wishes to thank the group GNSAGA of the “Istituto
Nazionale di Alta Matematica “Francesco Severi”” (INDAM) to which he belongs. We thank D. Macpherson for interesting discussions regarding weak elimination of imaginaries. We thank P. Boldrini, D. Carolillo and F. Pisciotta for useful comments in earlier drafts of this paper.}

\author{Gianluca Paolini}
\address{Department of Mathematics ``Giuseppe Peano'', University of Torino, Italy.}
\email{gianluca.paolini@unito.it}


\date{\today}
\maketitle


\section{Introduction}


	This paper sits at the intersection of model theory and descriptive set theory. Its purpose is to contribute to the determination of the complexity of a certain classification problem, in the sense of Borel reducibility. The  problem that we are interested in is the problem of classification of certain topological groups, known as oligomorphic groups, up to topological isomorphism, i.e., group isomorphisms that are also homeomorphisms. We now briefly introduce the basics of Borel reducibility and of oligomorphic groups, before discussing our results.
	
	Invariant descriptive set theory (see e.g. \cite{gao}) relies on the following crucial notion: given two Borel spaces $X$ and $Y$ and two equivalence relations $E_X$ and $E_Y$ on $X$ and $Y$, respectively, we say that $E_X$ is Borel reducible to $E_Y$ if there is a Borel map $f$ from $X$ to $Y$ such that for every $x_1, x_2 \in X$ we have that $x_1 E_X x_2$ if and only if $f(x_1) E_Y f(x_2)$. Now, if one has such a reduction, then one says that the equivalence relation $E_X$ is simpler than the equivalence relation $E_Y$. In particular, given two classification problems, if we are able to associate to  them corresponding Borel spaces, then we  are able to compare the two classification problems. Now, one might wonder: how does all this relate to actual classification problems in mathematics? It so happens that many mainstream classification problems can be ``coded'' via objects of an appropriate Baire space. Eminent examples of this phenomenon are for example the following three cases of classification problems (which are by now all completely understood, with respect to Borel reducibility): classification of ergodic measure preserving transformations \cite{foreman}, classification of separable $C^*$-algebras \cite{sabok},  classification of countable \mbox{torsion-free abelian groups up to isomorphism \cite{1205}.} 
	
	Now, if one identifies two equivalence relations if they are mutually reducible to each other (Borel bi-reducibility), then one can order equivalence relations in a partial order under Borel reducibility. In our paper we will actually only be interested in the bottom part of this partial order. At the bottom of this order there are {\em the smooth equivalence relations}, i.e., the equivalence relations of equality on a Borel space. As a Borel space can only be either finite, or of size $\aleph_0$, or of size $2^{\aleph_0}$, the smooth equivalence relations are often denoted by $=_n$, $=_{\mathbb{N}}$ and $=_{\mathbb{R}}$. Thus, an equivalence relation on a Borel space is said to be smooth if it is Borel reducible to one of these equivalence relations. In particular, if an equivalence relation has more than countably many equivalence classes, then saying that it is smooth is equivalent to saying that it is bi-reducible to $=_{\mathbb{R}}$. By a well-known dichotomy result due to Harrington-Kechris-Louveau \cite{E0} there is a simplest equivalence relation (up to Borel bi-reducibility) which is strictly above $=_{\mathbb{R}}$; this equivalence relation is denoted by $E_0$ and is defined, for $\eta, \theta \in \omega^\omega$, by letting $\eta E_0 \theta$ if and only if $\eta$ and $\theta$ are eventually equal. Another important equivalence relation (or rather equivalence relation up to Borel bi-reducibility) is $E_\infty$. $E_0$ has the property of being universal among countable Borel equivalence relations, in the sense that it is Borel-above any equivalence relation $E$ which is Borel and such that each $E$-equivalence class is countable. The existence of such an equivalence relation $E_\infty$ is due to Dougherty-Jackson-Kechris \cite{Einf}. $E_\infty$ has many concrete manifestations; one of them is for example the relation of isomorphism between finitely generated groups \cite{veli}.
	
	We now move to the introduction of oligomorphic groups. These groups are exactly the automorphism groups of countable $\omega$-categorical structures, i.e., those structures $M$ such that $M$ is the only countable model of its first-order theory up to isomorphism; where the topology associated to these groups is defined using as basic open subgroups the pointwise stabilizers of finite sets of $M$. This topology is known as a non-archimedean Polish group topology. Any such group can be topologically embedded in the infinite symmetric group $\mrm{Sym}(\omega)$, and from this perspective, the oligomorphic groups are exactly the ones that have only finitely many orbits on $\omega^n$, hence the name {\em oligomorphic}. Oligomorphic groups and $\omega$-categorical structures are objects of intense study in logic. One of the most interesting aspects of this area is that there are many correspondences between model-theoretic properties of $\omega$-categorical structures $M$ and group theoretic properties of $\mrm{Aut}(M)$. One of the most important results in this direction is the well-known Coquand-Ahlbrandt-Ziegler Theorem \cite{biinte} establishing that topological isomorphism between automorphism groups of countable $\omega$-categorical structures corresponds exactly to bi-interpretability of the underlying structures. Another important result is due to Rubin \cite{rubin}, which, under the additional assumption of admissibility of $\forall \exists$-interpretation and of no algebraicity, achieves a stronger result, i.e., $\mrm{Aut}(M)$ and $\mrm{Aut}(M)$ are isomorphic if and only if $M$ and $N$ are bi-definable. A result analogous to Rubin's was proved by S. Shelah and the author \cite{ssip_reconstruction} under the assumption of strong small index property (see below concerning this property).  In the context of $\omega$-categorical structures the notion of algebraicity is as in Galois theory: $a \in M$ is said to be algebraic over  a subset $A$ of $M$, if the orbit of $a$ under the pointwise stabilizer of $A$ is finite.

\medskip 

	We now move to the problem which is the object of interest of our paper. In \cite{nies, coarse} the authors start a systematic analysis of the complexity of topological isomorphism between oligomorphic groups. The main result of \cite{coarse} is that the relation of topological isomorphism between oligomorphic groups  is reducible to the already mentioned universal countable Borel equivalence relation $E_\infty$. Nonetheless, practically nothing is known on lower bounds for the complexity of this relation, in fact the only thing which is known is that equality on the reals is reducible to this relation, which is simply another way of saying that there are continuum many isomorphism types of topological groups of $\omega$-categorical structures, which has been known at least since 1994 (cf. \cite{henson_graphs, rubin}). The main problem is of course if we can reduce $E_0$ to topological isomorphism on oligomorphic groups. In this paper we prove that under the model-theoretic assumption of {\em weak elimination of imaginaries} this trivial lower bound $=_{\mathbb{R}}$ is actually sharp, as so groups having weak elimination of imaginaries cannot be used to try to reduce $E_0$.

	\begin{theorem}\label{main_theorem} The topological isomorphism relation between automorphism groups of countable $\omega$-categorical structures with weak elimination of imaginaries is smooth.
\end{theorem}
	
  Notice that in order for Theorem~\ref{main_theorem} to make sense we need  that automorphism groups of countable $\omega$-categorical structures with weak elimination of imaginaries form a Borel subset of the Borel space of oligomorphic groups, this is shown in \ref{prop_Borel}.
  
  Clearly, in order to argue in favor of the relevance of Theorem~\ref{main_theorem} to the open question from \cite{coarse} mentioned above, one has to argue about the significance of the assumption of weak elimination of imaginaries in the context of countable $\omega$-categorical structures. A first-order structure $M$ is said to have weak elimination of imaginaries if for any formula $\theta(\bar{x}, \bar{y})$ and $\bar{a} \in M^{\mrm{lg}(\bar{y})}$ there is a smallest finite algebraically closed set $A \subseteq M$ such that $\theta(\bar{x}, \bar{a})$ is equivalent to a formula with parameters in $A$ (cf. \cite[pg. 321]{poizat} and \cite[pg. 161]{hodges}). In the case of countable $\omega$-categorical structures this condition can be phrased in the language of permutation groups simply as: for every finite algebraically closed sets $A, B \subseteq M$, $G_{(A \cap B)} = \langle G_{(A)} \cup G_{(B)} \rangle_G$, where $G = \mrm{Aut}(M)$ and $G_{(C)}$ denotes the pointwise stabilizer of $C \subseteq M$. This latter condition in turn can be phrased as: for every $H \leq G = \mrm{Aut}(M)$ which contains the pointwise stabilizer of a finite set, there is a unique smallest finite algebraically closed set $K \subseteq M$ such that the following holds:
$$G_{(K)} \leq H \leq G_{\{K\}}.$$
This is the equivalent definition of weak elimination of imaginaries that we use in our paper (see e.g. \cite[Lemma~1.3]{evans_finite_cover} for a proof of these equivalences). 
Now, the assumption of (weak) elimination of imaginaries is often encountered in model theory, and particularly so in the context of $\omega$-categorical structures. Most of the familiar $\omega$-categorical structures {\em do} have weak elimination of imaginaries, thus so do free homogeneous structure in a finite relational language (cf. \cite{mc+tent} and \cite{ssip_canonical_hom}) and many other structures (also {\em with} algebraicity). The easiest example of an $\omega$-categorical structure which fails weak elimination of imaginaries is the theory of an equivalence relation with two infinite equivalence classes. It is easy to modify this example (e.g. using Henson digraphs) in order to show that there are continuum many non topologically isomorphic oligomorphic permutation groups which fail weak elimination of imaginaries, but no example of the failure of weak elimination of imaginaries which does not count as an ``obvious variation'' of the example above is known to the author. In this respect, our Theorem~\ref{main_theorem} seems to be very relevant to the open question from \cite{coarse} as it tells us: on one hand where to look for candidates for a possible reduction of $E_0$ to the topological isomorphism relation between oligomorphic permutation groups; and on the other hand it explains why so little is known on lower bounds for this relation, as many of the familiar structures do have weak elimination of imaginaries. 

	The assumption of weak elimination of imaginaries has been often considered in the literature in combination with another property, the so-called {\em small index property} (SIP) (see e.g. the celebrated paper \cite{lascar_hodges_shelah}). When a countable structure $M$ has SIP every isomorphism between $\mrm{Aut}(M)$ and $\mrm{Aut}(N)$ is automatically continuous and so the topological structure of $\mrm{Aut}(M)$ can be recovered from the algebraic structure of $\mrm{Aut}(M)$. In the literature the combination of small index property and weak elimination of imaginaries is referred to as the {\em strong small index property} (SSIP). So our work applies to this context and in particular in this case it allows us to obtain purely algebraic information on $\mrm{Aut}(M)$. But clearly our analysis has relevance also outside of the SIP context, as in \ref{main_theorem}.
Finally, notice also that the assumption of weak elimination of imaginaries is often made also in the context of \mbox{ergodic theory of automorphisms groups and related studies, cf. \cite{dyn1, dyn2}.}
	
	\smallskip
\noindent

	The ideas and techniques behind Theorem~\ref{main_theorem} are actually purely model-theoretic, they are of interest also outside of the context of \cite{coarse}. In a way, they can be seen as a ``reconstruction result'' {\em \`a la} Rubin (cf. \cite{rubin}) but {\em with} algebraicity. In fact, a purely model theoretic way of stating our result is the following:

	\begin{theorem}\label{second_theorem} There is a canonical assignment associating to an arbitrary countable structure $M$ a structure $\mathcal{E}^{\mrm{ex}}_M$ such that if $M, N$ have locally finite Galois-algebracity, they are Galois-homogeneous and they satisfy the Galois-sandwich condition (i.e., $M, N$ satisfy the three conditions in Hypothesis \ref{hyp}), then we have:
	\begin{enumerate}[(1)]
	\item every topological isomorphism $\alpha$ from $\mrm{Aut}(M)$ onto $\mrm{Aut}(N)$ induces an isomorphism $f_{\alpha}$ from $\mathcal{E}^{\mrm{ex}}_M$ onto $\mathcal{E}^{\mrm{ex}}_N$, and conversely;
	\item the assignment $\alpha \mapsto f_{\alpha}$ is a group isomorphism from the group of topological automorphisms of $\mrm{Aut}(M)$ onto the automorphism group $\mrm{Aut}(\mathcal{E}^{\mrm{ex}}_M)$.
	\end{enumerate}
\end{theorem}

	The way \ref{second_theorem} follows from \ref{main_theorem} is that we map $G$ to some $M$ with $G = \mrm{Aut}(M)$ and pass to a sequence of $\omega$-categorical structures $\mathcal{E}^{\mrm{ex}}_M(k)$, for $k < \omega$, which are approximations to the structure $\mathcal{E}^{\mrm{ex}}_M$ from \ref{second_theorem}, and then essentially apply \ref{second_theorem} to conclude that topological isomorphism is smooth. The structure $\mathcal{E}^{\mrm{ex}}_M$ can be considered as a model theoretic version of the coarse groups from \cite{coarse}. More concretely, the structure $\mathcal{E}^{\mrm{ex}}_M$ from Theorem~\ref{second_theorem} is a natural generalization of the orbital structure $\mathcal{E}_M$ associated to a first-order structure $M$, i.e., the structure with the same domain as $M$ and with equivalence relations $E_n$ such that, for $\bar{a}$ and $\bar{b}$ tuples from $M$ of length $n$, $\bar{a} E_n \bar{b}$ if and only if there is automorphism of $M$ which maps $\bar{a}$ to $\bar{b}$. The structure we introduce should appear rather natural to a model theorist, in fact it is simply made of triples $(K, p, K')$, where $K$ and $K'$ are finite Galois-algebraically closed (or simply algebraically closed, in the $\omega$-categorical context) subsets of $M$ and $p: K \cong K'$. Furthermore, the apparently exotic assumption which we termed ``Galois-sandwich condition'' is nothing but the straightforward abstraction in the context of permutation group theory of one of the equivalent characterizations of weak elimination of imaginaries in the $\omega$-categorical context (see in particular \cite[pg. 146]{hodges} and surrounding discussion), so the theorem obviously \mbox{applies to $\omega$-categorical structures with weak elimination of imaginaries.}

	Our methods have also concrete applications to the purely group theoretic problem of characterization of automorphisms of automorphism groups of countable homogeneous structures (and so in particular of description of outer automorphisms of such groups). The main application is in the study of classical groups (cf. \cite{evans}). In this respect, we prove the analog in the infinite dimensional context of the classical description of $\mrm{Aut}(\mrm{GL}_n(F))$ due to Schreier and van der Waerden \cite{vector}. The problem of description of the automorphism group of a general linear group over a ring has a long and established tradition (see e.g. \cite{reiner} or Dieudonn{\'e}'s monograph~\cite{die_book}).
	
	\begin{theorem}\label{vector_spaces} Let $V$ a vector space of dimension $\aleph_0$ over a finite field $F$. Then:
	$$\mrm{Aut}(\mrm{GL}(V)) \cong \mrm{Aut}(F^\times) \rtimes \mrm{Aut}(\mrm{P}(V)),$$
where we denote by $\mrm{P}(V)$ the projective space associated to the space $V$ (i.e., the set of subspaces of $V$ ordered by inclusion) and by $F^\times$ 
the multiplicative group of~$F$.
\end{theorem}
	Notice that the structure from the previous theorem {\em does} have algebraicity, which should be considered as a virtue of our approach, since most of the results in the area work under the assumption of {\em no algebraicity}. Our methods also have applications on structures with no algebraicity:


	\begin{theorem}\label{no_alge} Let $M$ be a countable $\omega$-categorical homogeneous structure with weak elimination of imaginaries and no algebraicity, then:
	$$\mrm{Aut}_{\mrm{top}}(\mrm{Aut}(M)) \cong \mrm{Aut}(\mathcal{E}_M),$$
where $\mathcal{E}_M$ is the orbital structure associated with $M$. In particular, if $M$ has in addition the small index property, then $\mrm{Aut}(\mrm{Aut}(M)) \cong \mrm{Aut}(\mathcal{E}_M)$. 
\end{theorem}

	\begin{corollary}\label{free_relational} If $M$ is a countable free homogeneous structure in a finite relational language, then we have that $\mrm{Aut}(\mrm{Aut}(M)) \cong \mrm{Aut}(\mathcal{E}_M)$.
\end{corollary}

	\begin{corollary}\label{rationals} Let $\mathbb{Q}$ be the countable dense linear order without endpoints. Then:
$$\mrm{Aut}(\mrm{Aut}(\mathbb{Q})) \cong \mrm{Aut}(\mathbb{Q}) \rtimes \mathbb{Z}_2.$$
\end{corollary}

	\begin{corollary}\label{outer_graphs} Denoting by $K_m$ a finite clique of size $m \geq 3$, we have:
	\begin{enumerate}[(1)]
	\item the $K_m$-free random graph $R(K_m)$ satisfies $\mrm{Out}(\mrm{Aut}(R(K_m))) \cong \{e \}$;
	\item\label{outer_graphs_2} the $k$-colored random graph $R_k$ ($k \geq 2$) satisfies $\mrm{Out}(\mrm{Aut}(R_k)) \cong \mrm{Sym}(k)$.
	\end{enumerate}
\end{corollary}

	Notice that Corollary~\ref{outer_graphs}(\ref{outer_graphs_2}) is explicitly stated in \cite{rubin}. Furthermore, in \cite{cameron_tarzi} it is proved that the exact sequence $1 \rightarrow \mrm{Aut}(R_m) \rightarrow \mrm{Aut}(\mrm{Aut}(R_m)) \rightarrow \mrm{Sym}(m) \rightarrow 1$ splits if and only if $m$ is odd. We do not know if Theorems~\ref{vector_spaces} and \ref{rationals} are known or explicitly stated somewhere. In any case, one of the main conceptual points of this paper is to frame all these results on outer automorphisms of automorphism groups of countable homogeneous structures under a single umbrella. Further, we believe that with our methods in principle it is possible to achieve an explicit description of $\mrm{Out}(\mrm{Aut}(M))$ for any $\omega$-categorical homogeneous structure with SSIP.
	

	\medskip
	\noindent We finish this introduction by giving an idea of the proof of Theorems~\ref{main_theorem} and \ref{second_theorem}. The fundamental idea has a long tradition, which is also mentioned at the end of the introduction of Rubin's paper \cite{rubin}, where he refers to works of Dixon, Neumann and Thomas \cite{dixon}, Truss \cite{truss}, and of Droste, Holland and Macpherson \cite{droste}. The idea is to use weak elimination of imaginaries in order to characterize those open subgroups of $\mrm{Aut}(M)$ which are stabilizers of singletons. As we will see, under the assumption of weak elimination of imaginaries, the open subgroups of $\mrm{Aut}(M)$ are exactly what we call the {\em generalized pointwise stabilizers} of $M$, i.e., those subgroups of the form $G_{(K, L)} = \{ f \in \mrm{Aut}(M) : f \restriction K \in L \}$, for $K$ a finite algebraically closed subset of $M$ and $L \leq \mrm{Aut}(K)$. If $M$ has no algebraicity (or simply $\mrm{acl}_M(a) = \{a\}$), then the stabilizers of singletons are simply the open proper subgroups of $\mrm{Aut}(M)$ which are maximal under inclusion, and this in turn reduces the understanding of topological isomorphisms from $\mrm{Aut}(M)$ to the problem of understanding of the orbital structure associated with $M$. This very same idea was exploited in \cite{ssip_reconstruction} in order to obtain the analog of Rubin's theorem in the content of SSIP mentioned above. But in the presence of algebraicity the situation is considerably more complicated, and no established technology to deal with this problem was known to the author. The main contribution of this paper is the introduction of new techniques and ideas that are tailored exactly for this task (taking inspiration from our paper \cite{ssip_reconstruction}). We hope that these techniques will have applications also on other problems concerning $\omega$-categorical structures, e.g. in the analysis of reducts of homogeneous $\omega$-categorical structures, a topic which received considerable attention recently (see e.g. \cite{pinsker}).
	
	\medskip
	\noindent The structure of the paper is as follows. In Section 2 we introduce the necessary definitions and facts and do the groundwork, towards a proof of Theorem~\ref{main_theorem}. In Section~3 we focus on applications of our methods to the problem of description of outer automorphisms of automorphisms \mbox{groups of countable $\omega$-categorical structures.}

\section{The proof of the main theorem}

	In this section we prove Theorem~\ref{main_theorem}. First of all, in the next two items we introduce the basic definitions from permutation group theory that we need.
	
\smallskip

	\begin{notation} 
	\begin{enumerate}[(1)]
	\item We write $A \subseteq_\omega M$ to denote finite subsets.
	\item Given a structure $M$ and $A \subseteq M$, and considering $\mrm{Aut}(M) = G$ in its natural action on $M$, we denote the pointwise (resp. setwise) stabilizer of $A$ under this action by $G_{(A)}$ (resp. $G_{\{ A \}}$). 
	\item We denote the subgroup relation by $\leq$.
\end{enumerate}
\end{notation}

\begin{definition}\label{algebraicity} Let $M$ be a structure and $G = \mrm{Aut}(M)$.
	\begin{enumerate}[(1)]
	\item $a \in M$ is Galois-algebraic  over $A \subseteq M$ if the orbit of $a$ under $G_{(A)}$ is finite. 
	\item The {\em Galois-algebraic closure} of $A \subseteq M$, denoted as $\mrm{acl}^{\mrm{g}}_M(A)$, is the set of elements of $M$ which are Galois-algebraic over $A$.
	\item $a \in M$ is Galois-definable  over $A \subseteq M$ if the orbit of $a$ under $G_{(A)}$ is $\{a\}$. 
	\item The {\em Galois-definable closure} of $A \subseteq M$ in $M$, denoted as $\mrm{dcl}^{\mrm{g}}_M(A)$, is the set of elements of $M$ which are Galois-definable over $A$.
\end{enumerate}
\end{definition}


	In the next hypothesis we isolate the abstract properties of oligomorphic groups that make our ``reconstruction proof'' work. We could have worked at the less general level of oligomorphic groups directly but we hope (and believe) that the added generality might be relevant for other applications of our methods.

	\begin{hypothesis}\label{hyp} Throughout this section, let $M$ be countable and such that:
	\begin{enumerate}[(1)]
	\item\label{hyp_1} $M$ has locally finite Galois-algebraicity, i.e., \mbox{for every $A \subseteq_\omega M$, $|\mrm{acl}^{\mrm{g}}_M(A)| < \aleph_0$;}
	\item\label{hyp_3} for every $K_1, K_2 \subseteq M$ which are finite and Galois-algebraically closed and for every isomorphism $f: K_1 \cong K_2$, there is an extension of $f$ to an automorphism of~$M$;
	\item\label{hyp_2} for every $H \leq G = \mrm{Aut}(M)$ which is open (in the $\mrm{Aut}(M)$ topology), there is a unique finite Galois-algebraically closed set $K \subseteq M$ such that:
$$G_{(K)} \leq H \leq G_{\{K\}}.$$
	\end{enumerate}
\end{hypothesis} 

	The next definition is one of a technical nature, but an important one, in fact we need this for the statement of the main technical result of this paper, i.e., \ref{the_first_crucial_lemma}.

	\begin{notation}\label{kM_notation} Given $M$ as in \ref{hyp} we define $k_M$ as the following cardinal:
	$$k_M = \mrm{sup} \{k < \omega : \exists a_1, ..., a_k \in M \text{ s.t. } \mrm{acl}^{\mrm{g}}_M(a_1) \subsetneq \cdots \subsetneq \mrm{acl}^{\mrm{g}}_M(a_k)\}.$$
\end{notation}

	Notice that if $M$ is $\omega$-categorical, then the cardinal number $k_M$ from \ref{kM_notation} is finite. To see this, suppose not, then for every $k < \omega$ we can find $a_k$ and $(b^k_i : 1 \leq i < k)$ such that $\mrm{acl}^{\mrm{g}}_M(b^k_1) \subsetneq \cdots \mrm{acl}^{\mrm{g}}_M(b^k_{k-1}) \subsetneq \mrm{acl}^{\mrm{g}}_M(a_k)$, and so there is no bound on algebraic closures of singletons, which is impossible (cf. e.g. \cite[Fact~2.4]{hodges_al}).
	
	\begin{remark}\label{remark_cat} 
	\textrm{If $M$ is $\omega$-categorical and homogeneous, then $M$ satisfies items (\ref{hyp_1}) and (\ref{hyp_3}) of Hypothesis~\ref{hyp}. In particular, if $G$ is oligomorphic, then the canonical structure $M_G$ associated to $G$ obtained adding predicates naming orbits (cf. e.g. \cite[pg. 136]{hodges}) satisfies Hypothesis~\ref{hyp}(\ref{hyp_1})(\ref{hyp_3}). Furthermore, if on $M$ the operator $\mrm{acl}^{\mrm{g}}$ satisfies the condition $a \in \mrm{acl}^{\mrm{g}}(b) \setminus \mrm{acl}^{\mrm{g}}(\emptyset)$ implies $b \in \mrm{acl}^{\mrm{g}}(a)$, then $k_M \leq 2$ and it is equal to $1$ exactly when $G = G_{(\mrm{acl}(\emptyset))}$. In particular, if $\mrm{acl}^{\mrm{g}}_M(a) = \{a\}$ for every $a \in M$, then $k_M = 1$. Finally, as explained in the introduction, in the context of $\omega$-categorical structure Hypothesis~\ref{hyp}(\ref{hyp_2}) is exactly weak elimination of imaginaries. }
\end{remark}

	We now introduce what we refer to as the age of $M$ ($\mathbf{A}(M)$) and the {\em expanded} age of $M$ ($\mathbf{EA}(M)$), where the terminology takes inspiration from Fra\"iss{\'e} theory.

	\begin{notation} 
	\begin{enumerate}[(1)]
	\item We let $\mathbf{A}(M) = \{ \mrm{acl}^{\mrm{g}}_M(B) : B \subseteq_{\omega} M \}$.
	\item We let $\mathbf{EA}(M) = \{ (K, L) : K \in \mathbf{A}(M) \text{ and } L \leq \mrm{Aut}(K) \}$.
\end{enumerate}
\end{notation}  

	One of the crucial ideas behind our reconstruction theorem is that, in a sense, all the information needed for reconstruction is encoded in the combinatorics of the lattice of Galois algebraically closed sets of $M$. But in order to argue toward smoothness of the topological isomorphism relation we have to cut the lattice of Galois algebraically closed sets of $M$ at each level $k < \omega$, where level is meant in terms of distance from the minimum element. \mbox{This is the point of the next definition.}

\begin{notation}\label{the_lattice notation} We let $\mathbf{A}_+(M)$ to be the set $\mathbf{A}(M)$ extended with an extra element $\mrm{dcl}^{\mrm{g}}_M(\emptyset)$ (of course if $\mrm{dcl}^{\mrm{g}}_M(\emptyset) = \mrm{acl}^{\mrm{g}}_M(\emptyset)$, then $\mathbf{A}_+(M) =\mathbf{A}(M)$). On $\mathbf{A}_+(M)$ we consider a lattice structure by letting $0 = \mrm{dcl}^{\mrm{g}}_M(\emptyset) \leq \mrm{acl}^{\mrm{g}}_M(\emptyset)$ be at the bottom of the lattice and by defining (as usual) $K_1 \wedge K_2 = K_1 \cap K_2$ and $K_1 \vee K_2 = \mrm{acl}^{\mrm{g}}(K_1 \cup K_2)$. For every $0 < k < \omega$, we define $\mathbf{A}_k(M)$ to be the set of $K \in \mathbf{A}(M)$ such that $\mrm{dcl}^{\mrm{g}}_M(\emptyset) \neq K$ and $K$ is at distance $\leq k$ from $0 = \mrm{dcl}^{\mrm{g}}_M(\emptyset)$ in the Hasse diagram of the lattice $(\mathbf{A}_+(M), \wedge, \vee)$ just defined. Also, by convention, we let $\mathbf{A}_\omega(M) = \mathbf{A}(M)$. 
\end{notation}

	In the next definition we introduce one of the most important definitions of the paper, we refer to the objects we introduce as {\em generalized pointwise stabilizers}.

	\begin{definition} Let $(K, L) \in \mathbf{EA}(M)$, we define:
	$$ G_{(K, L)} = \{ f \in \mrm{Aut}(M) : f \restriction K \in L \}.$$
\end{definition}

Notice that if $L = \{ \mrm{id}_K \}$, then  $G_{(K, L)} = G_{(K)}$, i.e., it equals the pointwise stabilizer of $K$, and that if $L = \mrm{Aut}(K)$, then $G_{(K, L)} = G_{\{ K \}}$, i.e., it equals the setwise stabilizer of $K$. That is why we referred to $ G_{(K, L)}$ as generalized pointwise stabilizers (and why we use $\mathcal{GS}(M)$ below). We introduce the following notation:
$$\mathcal{PS}(M) = \{ G_{(K)} : K \in \mathbf{A}(M) \} \; \text{ and } \; \mathcal{GS}(M) =\{ G_{(K, L)} : (K, L) \in \mathbf{EA}(M) \}.$$

\medskip

	One of the crucial ingredients of our proof is the following lemma, showing that under the assumptions from Hypothesis~\ref{hyp}, the open subgroups of $\mrm{Aut}(M) = G$ are exactly 
the generalized pointwise stabilizers that we introduced right above.

	\begin{lemma}\label{char_subgr_small_index} $\{ H \leq G : H \text{ is open} \} = \mathcal{GS}(M)$.
\end{lemma}

	\begin{proof} The containment from right to left is trivial. Let $H \leq G$ be open. By \ref{hyp}(\ref{hyp_2}), there is a unique smallest finite Galois-algebraically closed set $K \subseteq M$ such that $G_{(K)} \leq H \leq G_{\{K\}}$. Firstly, we claim that $G_{(K)} \trianglelefteq G_{\{ K \}}$. In fact, for $g \in G_{\{ K \}}$, $h \in G_{(K)}$ and $a \in K$, we have $ghg^{-1}(a) = gg^{-1}(a) = a$, since $g^{-1}(a) \in K$ and $h \in G_{(K)}$. Secondly, for $g, h \in G_{\{ K \}}$, we have $g^{-1}h \in G_{(K)}$ iff $g \restriction K = h \restriction K$. Hence, the map $f: gG_{(K)} \mapsto g \restriction K$, for $g \in G_{\{ K \}}$, is such that: 
	\begin{equation}\tag{$\star$}\label{equation_label}
	f:G_{\{ K \}} / G_{(K)}  \cong  \mrm{Aut}(K),
\end{equation}
since every $f \in \mrm{Aut}(K)$ extends to an automorphism of $M$ (recall \ref{hyp}(\ref{hyp_3})). Thus, by the fourth isomorphism theorem we have $H = G_{(K, L)}$ for $L = \{ f \restriction K : f \in H \}$.
\end{proof}

	With Lemma~\ref{char_subgr_small_index} at hand we know much but we want to do more, i.e., we want to give an algebraic characterization of the pointwise stabilizers among the generalized pointwise stabilizers. This is the content of the next two propositions.

	\begin{proposition}\label{normality_prop} Let $H_1, H_2 \in \mathcal{GS}(M)$. The following conditions are equivalent:
	\begin{enumerate}[(1)]
	\item $H_1 \trianglelefteq H_2$ and $[H_2: H_1] < \omega$;
	\item there is $K \in \mathbf{A}(M)$ and $L_1 \trianglelefteq L_2 \leq \mrm{Aut}(K)$ such that $H_i = G_{(K, L_i)}$ for $i =1,2$.
\end{enumerate}
\end{proposition}

	\begin{proof} Concerning ``(2) implies (1)'', by the normality of $L_1$ in $L_2$ we have that, for $g \in G_{(K, L_2)}$ and $h \in G_{(K, L_1)}$, $g h g^{-1} \restriction K \in L_1$, while the fact that $[H_2: H_1] < \omega$ follows from the proof of Lemma \ref{char_subgr_small_index}. We now show that ``(1) implies (2)''. By assumption, $H_i = G_{(K_i, L_i)}$ for $(K_i, L_i) \in \mathbf{EA}(M)$ ($i =1,2$). 
	\begin{enumerate}[$(*)_1$]
	\item $K_2 \subseteq K_1$.
\end{enumerate}
Suppose not, and let $a \in K_2 - K_1$ witness this. Then we can find $f \in G$ such that $f \restriction K_1 = \mrm{id}_{K_1}$ and $f(a) \not\in K_2$. It follows that $f \in H_1 - H_2$, a contradiction.
\begin{enumerate}[$(*)_2$]
	\item $K_1 \subseteq K_2$.
\end{enumerate}
Suppose not, as $K_2$ is Galois-algebraically closed and $M$ is infinite, we can find $f_n \in G$, for $n < \omega$, such that $f_n \restriction K_2 = \mrm{id}_{K_2}$, and in addition the sets $\{ f_n(K_1 - K_2) : n < \omega \}$ are pairwise disjoint. Then clearly, for every $n < \omega$, $f_n \in H_2$ and $\{ f_nH_1 : n < \omega \}$ are distinct, contradicting the assumption $[H_2: H_1] < \omega$.
\begin{enumerate}[$(*)_3$]
	\item $L_1 \leq L_2$.
\end{enumerate}
Suppose not, and let $h \in L_1 - L_2$. Then $h$ extends to an automorphism $f$ of $M$. Clearly $f \in H_1 - H_2$, a contradiction.
\begin{enumerate}[$(*)_4$]
	\item $L_1 \trianglelefteq L_2$.
\end{enumerate}
Suppose not, and let $g_i \in L_i$ ($i = 1, 2$) be such that $g_2g_1g_2^{-1} \not\in L_1$. Then $g_i$ extends to an automorphism $f_i$ of $M$ ($i = 1, 2$). Clearly $f_i \in H_i$ ($i = 1, 2$), and $f_2f_1f_2^{-1} \not\in H_1$, a contradiction.
\end{proof}

	\begin{proposition}\label{char_point_stab}
	$$\mathcal{PS}(M) = \{ H \in \mathcal{GS}(M) : \not\!\exists H' \in \mathcal{GS}(M) \text{ with } H' \subsetneq H, H' \trianglelefteq H \text{ and } [H : H'] < \omega \}.$$
\end{proposition}

	\begin{proof} First we show the containment from left to right. Let $H_2 \in \mathcal{PS}(M)$ and assume that there exists $H_1 \in \mathcal{GS}(M)$ such that $H_1 \subsetneq H_2, H_1 \trianglelefteq H_2 \text{ and } [H_2: H_1] < \omega$. By Proposition \ref{normality_prop}, $H_i = G_{(K_i, L_i)}$ for $(K_i, L_i) \in \mathbf{EA}(M)$ ($i =1,2$) and $K_1 = K = K_2$. Now, as $H_2 \in \mathcal{PS}(M)$, $L_2 = \{ \mrm{id}_K \}$. Hence, $L_1 = L_2$, and so $H_1 = H_2$, a contradiction. We now show the containment from right to left. Let $H \in \mathcal{G}_2$, then $H = G_{(K, L)}$ for $(K, L) \in \mathbf{EA}(M)$. If $L \neq \{ \mrm{id}_K \}$ then letting $H' = G_{(K, \{ \mrm{id}_K \})}$ we have $H' \subsetneq H$, $H' \trianglelefteq H$ and $[H : H'] < \omega$, a contradiction.
\end{proof}

	Once we have characterized $\mathcal{PS}(M)$ in $\mathcal{GS}(M)$, it is immediate to characterize those elements of $\mathcal{PS}(M)$ which correspond to $\mathbf{A}_k(M)$, recalling Notation~\ref{the_lattice notation}.

	\begin{proposition}\label{char_point_stab_singletons} For $0 < k < \omega$, we let
	$\mathcal{PS}_k(M)$ be the set of  $H \in \mathcal{PS}(M)$ such that $H \neq G$ and any chain of elements of $\mathcal{PS}(M)$ starting at $H$ and arriving at $G$ has length $\leq k$. Also, by convention, we let $\mathcal{PS}_\omega(M) = \mathcal{PS}(M)$. Then, recalling the notation introduced in \ref{the_lattice notation} (so $\mathbf{A}_\omega(M) = \mathbf{A}(M)$), we have the following:
	$$\mathcal{PS}_k(M) = \{G_{(K)} : K \in \mathbf{A}_k(M) \}.$$
\end{proposition}

	\begin{proof} This is obvious, recalling \ref{char_point_stab}.
\end{proof}

	The following proposition is not needed for the proof of Theorem~\ref{main_theorem} but we believe that it is of independent interest and also it will be used in Section 3.

\begin{proposition}\label{char_Lpoint_stab} Let $L$ be a finite group and $H \in \mathcal{PS}(M)$. Then TFAE:
	\begin{enumerate}[(1)]
	\item $H = G_{(K)}$ and $\mrm{Aut}(K) \cong L$;
	\item there is $H' \in \mathcal{GS}(M)$ such that $H \trianglelefteq H'$, $[H': H] < \omega$, $H'$ is maximal under these conditions and $H'/H \cong L$.
\end{enumerate}
\end{proposition}

	\begin{proof} Concerning the implication ``(1) implies (2)'', let $H' = G_{\{ K \}}$, then, by Proposition \ref{normality_prop} and equation (\ref{equation_label}) in the proof of Lemma \ref{char_subgr_small_index}, we have that $H'$ is as wanted. Concerning the implication ``(2) implies (1)'', if $H$ and $H'$ are as in (2), then, by Proposition \ref{normality_prop} and equation (\ref{equation_label}) in the proof of Lemma \ref{char_subgr_small_index}, it must be the case that $H' = G_{\{ K \}}$ and $H = G_{(K)}$ for some $K \in \mathbf{A}(M)$ s.t. $\mrm{Aut}(K) \cong L$.
\end{proof}

	We now introduce a first-order structure which we refer to as the {\em expanded structure of $M$ of depth $\leq k$} (for $k \leq \omega$), where $k$ has to be thought of in the context of Notation~\ref{the_lattice notation}. Essentially, it is the orbital structure on the set of Galois algebraically closed sets of $M$ which are at distance $\leq k$ from the bottom of the lattice.

\begin{definition}\label{def_expanded_1}  Let $0 < k \leq \omega$ and recall the notations from \ref{the_lattice notation}.
\begin{enumerate}[(1)]
	\item We let $M^{\mrm{ex}}_k$ be $\{(K, p, K') : K, K' \in \mathbf{A}_k(M) \text{ and } p: K \cong K'\}.$
	\item\label{the_identification} We identify $\mathbf{A}_k(M)$ with the set of triples $(K, \mrm{id}_K, K)$.
\end{enumerate}
\end{definition}

	\begin{definition}\label{def_expanded_12} Let $0 < k \leq \omega$. We define a first-\mbox{order structure $\mathcal{E}^{\mrm{ex}}_M(k)$ as follows:}
	\begin{enumerate}[(1)]
	\item $\mathcal{E}^{\mrm{ex}}_M(k)$ has as domain $M^{\mrm{ex}}_k$;
	\item\label{def_expanded_12_unary} we define a unary predicate which holds of $\mathbf{A}_k(M)$, recalling \ref{def_expanded_1}(\ref{the_identification}); 
	\item\label{def_expanded_12_binary} for $1 < n< \omega$, we define an $n$-ary predicate $E_n$ as follows:
	$$((A_1, p_1, B_1), ..., (A_n, p_n, B_n)) \in E^{\mathcal{E}^{\mrm{ex}(k)}_M}_n$$
	$$\Updownarrow$$
	$$\exists g \in \mrm{Aut}(M) \text{ s.t. for all } i \in [1, n], \,  g(A_i) = B_i \text{ and } g \restriction A_i = p_i;$$
	\item\label{dom} we define a binary predicate $\mrm{Dom}$ (which stands for ``domain'')~such~that \newline $(A, p, B) \in \mathcal{E}^{\mrm{ex}}_M(k)$ is in relation with $C \in \mathbf{A}_k(M)$ if and only if $A = C$;
	\item\label{cod} we define a binary predicate $\mrm{Cod}$ (which stands for ``codomain'') such that $(A, p, B) \in \mathcal{E}^{\mrm{ex}}_M(k)$ is in relation with $C \in \mathbf{A}_k(M)$ if and only if $B = C$.
	
\end{enumerate}
When we write $\mathcal{E}^{\mrm{ex}}_M$ (so without the $k$) we mean $\mathcal{E}^{\mrm{ex}}_M(\omega)$.
\end{definition}

	The next theorem shows that as soon as $k \geq \mrm{max}\{k_M, k_N\}$ the problem of determination of $\mrm{Aut}(M) \cong_{\mrm{top}} \mrm{Aut}(N)$ reduces completely to the problem of determination of $\mathcal{E}^{\mrm{ex}}_M(k) \cong \mathcal{E}^{\mrm{ex}}_N(k)$. This theorem is the core of the paper. 

	\begin{theorem}\label{the_first_crucial_lemma} Let $M, N$ be as in \ref{hyp}. Then we have the following:
	\begin{enumerate}[(1)]
	\item $G: = \mrm{Aut}(M) \cong_{\mrm{top}} \mrm{Aut}(N):= H$ implies $\mathcal{E}^{\mrm{ex}}_M(k) \cong \mathcal{E}^{\mrm{ex}}_N(k)$, for all $0 < k \leq \omega$.
	\item If $\mrm{max}\{k_M, k_N\} \leq k$, then $\mathcal{E}^{\mrm{ex}}_M(k) \cong \mathcal{E}^{\mrm{ex}}_N(k)$ implies $G \cong_{\mrm{top}} H$.
	\end{enumerate}
\end{theorem}

	\begin{proof} Concerning (1), let $\alpha : \mrm{Aut}(M) \cong_{\mrm{top}} \mrm{Aut}(N)$. By \ref{char_point_stab_singletons}, $\alpha$ induces a bijection of  $\mathcal{PS}_k(M)$ onto $\mathcal{PS}_k(N)$. Let $f$ be the corresponding bijection of $\mathbf{A}_k(M)$ onto $\mathbf{A}_k(N)$. Let also $G = \mrm{Aut}(M)$ and $H = \mrm{Aut}(N)$. Then, \mbox{for $K \in \mathbf{A}_k(M)$ and $g \in  G$:}
$$\alpha(gG_{(K)}) = \alpha(g)\alpha(G_{(K)}) = \alpha(g) H_{(f(K))}.$$
Hence, $\alpha$ also induces a bijection of $\{gG_{(K)} : K \in \mathbf{A}_k(M)\}$ onto $\{hH_{(K)} : K \in \mathbf{A}_k(N)\}$. Thus, given $(K, p, K') \in M^{\mrm{ex}}_k$ we have that for any extension $\tilde{p}$ of $p$ to an automorphism of $M$ we have that $\alpha(\tilde{p} G_{(K)}) = \alpha(\tilde{p}) H_{(f(K))}$ and also that:
$$\alpha(\tilde{p}) H_{(f(K))} \alpha(\tilde{p})^{-1} = H_{f(K')},$$
and so it makes sense to define $f_\alpha(K, p, K') = f(K, p, K')$ as $(f(K), \alpha(\tilde{p}) \restriction f(K), f(K'))$ (it is easy to see that this does not depend on the choice of $\tilde{p}$, recalling the identification of $gG_{(K)}$ with the triple $(K, g \restriction K, g(K))$). Notice that $f(K, \mrm{id}_K, K)$ is $(f(K), \mrm{id}_{f(K)}, f(K))$, as $\tilde{\mrm{id}}_K G_{(K)} = G_{(K)}$ and so $\alpha(\tilde{\mrm{id}}_K) H_{(f(K))} = H_{(f(K))}$, hence our definition of $f$ is notationally consistent with the identification made in \ref{def_expanded_1}(\ref{the_identification}). Thus, the map $f = f_{\alpha}$ that we just defined is a bijection between $M^{\mrm{ex}}_k$ and $N^{\mrm{ex}}_k$ which sends $\mathbf{A}_k(M)$ onto $\mathbf{A}_k(N)$ (and so it preserves the unary predicate from \ref{def_expanded_12}(\ref{def_expanded_12_unary})). 
We now show that $f$ preserves also the relations $E_n$, for $1 < n < \omega$, from \ref{def_expanded_12}(\ref{def_expanded_12_binary}). To this extent observe the following sequence of equivalences:
$$((A_1, p_1, B_1), ..., (A_n, p_n, B_n)) \in E^{\mathcal{E}^{\mrm{ex}}_M(k)}_n$$
$$\Updownarrow$$
$$\exists g \in \mrm{Aut}(M) \text{ s.t. for all } i \in [1, n], \,  g(A_i) = B_i \text{ and } g \restriction A_i = p_i.$$
$$\Updownarrow$$
$$\exists g \in \mrm{Aut}(M) \text{ s.t. } g G_{(A_1)}g^{-1} = G_{(B_1)}, \, ... , \, g G_{(A_n)}g^{-1} = G_{(B_n)}$$
$$\Updownarrow$$
$$\exists g \in \mrm{Aut}(M) \text{ s.t. } \alpha(g) H_{(f(A_1))}\alpha(g)^{-1} = H_{(f(B_1))}, \,  ... , \, \alpha(g) H_{(f(A_n))}\alpha(g)^{-1} = H_{(f(B_n))}$$
$$\Updownarrow$$
$$f(A_1, p_1, B_1), \, ..., \, f(A_n, p_n, B_n) \in E^{\mathcal{E}^{\mrm{ex}}_N(k)}_n.$$
Finally the fact that also the predicates $\mrm{Dom}$ and $\mrm{Cod}$ from \ref{def_expanded_12}(\ref{dom})(\ref{cod}) are preserved is easy, or see the proof of \ref{second_theorem} where more is shown. So we proved (1).

\medskip
\noindent
Concerning item (2), let $\mrm{max}\{k_M, k_N\} \leq k$ and $f: \mathcal{E}^{\mrm{ex}}_M(k) \cong \mathcal{E}^{\mrm{ex}}_N(k)$. Recalling \ref{kM_notation}, observe that under this assumption we crucially have:
\begin{equation}\tag{$\star$}
\bigcup \{ K : K \in \mathbf{A}_k(M) \} = M \text{ and } \bigcup \{ K : K \in \mathbf{A}_k(N) \} = N.
\end{equation}
Now, given $g \in G$ we have to define $\alpha_f(g) = \alpha(g) \in H$ in such a way that the resulting map $\alpha_f: \mrm{Aut}(M) \cong_{\mrm{top}} \mrm{Aut}(N)$. To this extent, first of all enumerate $\mathbf{A}_k(M)$ as $(K_i : i < \omega)$ and let $g \in G$. For every $i < \omega$, let:
$$(K_i, g \restriction K_i, g(K_i)) =: (A_i, p_i, B_i).$$
Then, clearly, for every $n < \omega$, we have that:
$$((A_1, p_1, B_1), ..., (A_n, p_n, B_n)) \in E^{\mathcal{E}^{\mrm{ex}}_M(k)}_n,$$
and so, as $f: \mathcal{E}^{\mrm{ex}}_M(k) \cong \mathcal{E}^{\mrm{ex}}_N(k)$, we infer that:
$$f(A_1, p_1, B_1), \, ..., \, f(A_n, p_n, B_n) \in E^{\mathcal{E}^{\mrm{ex}}_N(k)}_n.$$
Now, for every $i < \omega$, let $h_i$ be the second component of the triple $f(A_1, p_1, B_1)$. Then, recalling how the predicates $E^{\mathcal{E}^{\mrm{ex}}_N(k)}_n$ were defined, for every $n < \omega$ there is $\tilde{h}_n \in H$ such that $h_i \subseteq \tilde{h}_n$, for every $i \leq n$. Now, by $(\star)$ we have that $\bigcup \{ K : K \in \mathbf{A}_k(N) \} = N$ and so for every $a \in N$ we have that the value $(\tilde{h}_n(a) : n < \omega)$ is eventually constant. Thus, we can define:
$$\alpha_f(g) = \alpha(g) = \mrm{lim}(\tilde{h}_n : n < \omega).$$
Clearly $\alpha(g)$ is a one-to-one function from $N$ into $N$. We now argue that $\alpha(g)$ is surjective, as it will then immediately follow that it is an automorphism of $N$, because of how it was defined. To this extent, notice that $f \restriction \mathbf{A}_k(M)$ induces a bijection of $\mathbf{A}_k(M)$ onto $\mathbf{A}_k(N)$ and, as before, by $(\star)$ we have that $\bigcup \{ K : K \in \mathbf{A}_k(N) \} = N$, using this it is easy to see that $\alpha(g)$ is surjective (recalling the predicates $\mrm{Dom}$ and $\mrm{Cod}$ from \ref{def_expanded_12}(\ref{dom})(\ref{cod})). Furthermore, it is immediate to verify that this correspondence $g \mapsto \alpha(g) = \alpha_f(g)$ is a topological isomorphism from $\mrm{Aut}(M)$ onto $\mrm{Aut}(N)$. This concludes the proof. 
\end{proof}

	The fact that above we used an arbitrary $k \leq \omega$ instead of using directly $\omega$ is explained by the following crucial lemma needed for the proof of Theorem~\ref{main_theorem}. On the other hand, recall that if $M$ is $\omega$-categorical, then $k_M < \omega$.

	\begin{lemma}\label{preserve_categoricity} If $G$ is oligomorphic, then, for every $k < \omega$, $\mathcal{E}^{\mrm{ex}}_{M_G}(k)$ is $\omega$-categorical.
\end{lemma}

	\begin{proof} Fix $k < \omega$ and let $\mathcal{E}^{\mrm{ex}}_{M_G}(k) = A_k$. We have to show that $\mrm{Aut}(A_k)$ acts oligomorphically on $A_k$. It follows from the proof of \ref{the_first_crucial_lemma}, that there is a bijective correspondence between $A_k$ and the set of cosets of open subgroups of the form $G_{(K)}$ for $K \in \mathbf{A}_k(M)$, and that $G$ acts on the latter set by conjugation as $\mrm{Aut}(M)$ acts on $M$. Hence, if for some $n < \omega$ the action of $\mrm{Aut}(M)$ on $(A_k)^n$  would have infinitely many orbits, the same would happen for the action of $G$ on cosets of open subgroups of the form $G_{(K)}$ for $K \in \mathbf{A}_k(M)$, but then, as the set $\{ |K| : K \in \mathbf{A}_k(M)\}$ is finite, we would be able to find $m < \omega$ such that $G$ acts on $M^m$ (by automorphisms) with infinitely many orbits, which is absurd.
\end{proof}

	The following proposition shows that the domain of our Borel reduction from Theorem~\ref{main_theorem} is Borel, which is \mbox{necessary for the statement of \ref{main_theorem} to make sense.}

	\begin{proposition}\label{prop_Borel} Having weak elimination of imaginaries is a Borel property of $\omega$-categorical structures, equivalently, automorphism groups of countable $\omega$-omega categorical structures with weak elimination of imaginaries form a Borel set.
\end{proposition}

	\begin{proof}  We argue as follows:
	\begin{enumerate}[(i)]
	\item Given a closed $G \leq \mrm{Sym}(\omega)$ we can in a Borel way obtain a countable structure $M_G$ in a countable \mbox{signature s.t. $G \cong_{\mrm{top}} \mrm{Aut}(M_G)$ (cf. \cite[Section~2.5]{nies});}
	\item The collection of $\omega$-categorical structures with weak elimination of imaginaries is a Borel subset of the Borel space of structures with domain $\omega$ in a language with $\aleph_0$-many predicates of arity $n$, for every $0 < n < \omega$.
	\item The inverse image of a Borel set under a Borel function is Borel.
	\item Oligomorphic groups are Borel in the space of closed subgroups of $\mrm{Sym}(\omega)$.
	\item Intersections of Borel sets are Borel.
\end{enumerate}
Item (iii) is by definition. Items (ii) and (v) are clear. Item (iv) is easy, see also \cite[pg. 2150029-7]{coarse}. Concerning (ii), we use the following definition of weak elimination of imaginaries from \cite[pg.~161]{hodges}: $A$ has weak elimination of imaginaries if for every equivalence formula $\theta(\bar{x}, \bar{y})$ of $A$ there are a formula $\varphi(\bar{x}, \bar{z})$ and a finite set of tuples $X$ from $A$ such that the equivalence class $\bar{a}/\theta$ can be written as $\varphi(A^n, \bar{b})$ iff $\bar{b}$ lies in $X$. But it is easy to see that this condition can be written as an $\mathfrak{L}_{\omega_1, \omega}$-theory (infinitary logic with countable conjunctions and countable disjunctions) and so clearly the set of structures satisfying this condition is Borel. Hence (ii) holds.
\end{proof}


	Finally, we put everything together in order to prove our smoothness result.

	\begin{proof}[Proof of Theorem~\ref{main_theorem}] The idea of the proof is similar to the proof of \cite[Proposition~4.2]{coarse}. As shown in \cite[pg.~1193]{nies}, given a closed subgroup $G$ of $\mrm{Sym}(\omega)$ we can in a Borel way obtain a countable structure $M_G$ in a countable signature such that $G \cong_{\mrm{top}} \mrm{Aut}(M_G)$. Similarly, for fixed $0 < k < \omega$, the assignment $M_G \mapsto \mathcal{E}^{\mrm{ex}}_{M_G}(k)$ is Borel, and so also the  assignment $G \mapsto (\mathcal{E}^{\mrm{ex}}_{M_G}(k) : 0 < k < \omega)$ is Borel.
	Now, by \ref{preserve_categoricity}, we have that for every $0 < k < \omega$ the structures $\mathcal{E}^{\mrm{ex}}_{M_G}(k)$ and $\mathcal{E}^{\mrm{ex}}_{M_H}(k)$ are $\omega$-categorical (as $G$ and $H$ are oligomorphic). Finally, for countable structures in a fixed countable language mapping $M$ to its first-order theory is also a Borel assignment. Also recall that for $G$ oligomorphic we have that the value $k_{M_G}$ from \ref{kM_notation}  is finite. Hence, if $M_G$ and $M_H$ have weak elimination of imaginaries, then they satisfy the assumptions of \ref{hyp} and so by Theorem~\ref{the_first_crucial_lemma} \mbox{we have that the following holds:}
	$$G \cong_{\mrm{top}} H \; \Leftrightarrow (\mrm{Th}(\mathcal{E}^{\mrm{ex}}_{M_G}(k)) : 0 < k < \omega) = (\mrm{Th}(\mathcal{E}^{\mrm{ex}}_{M_H}(k)) : 0 < k < \omega),$$
notice in fact that $M_G$ and $M_H$ are $\omega$-categorical and so, as observed right after \ref{kM_notation}, $k_M, k_N < \omega$. 
Thus, we Borel reduced the equivalence relation $\cong_{\mrm{top}}$ to equality on a standard Borel space, and so we have that the relation $\cong_{\mrm{top}}$ is smooth. 
\end{proof}


\section{Outer automorphisms} 

	In this section we deal with applications of our methods to the problem of description of outer automorphisms of automorphism groups of $\omega$-categorical structures. First of all, although the structure $\mathcal{E}^{\mrm{ex}}_M(k)$ that we defined in \ref{def_expanded_12} was sufficient for the sake of proving Theorem~\ref{main_theorem}, actually any isomorphism $\alpha: \mrm{Aut}(M) \cong_{\mrm{top}} \mrm{Aut}(N)$ preserve more relations than the ones isolated in $\mathcal{E}^{\mrm{ex}}_M(k)$. This motivates the introduction of a definitional expansion of $\mathcal{E}^{\mrm{ex}}_M(k)$, which we name $\mathcal{E}^{\hat{\mrm{ex}}}_M(k)$.  The reason for the introduction of $\mathcal{E}^{\hat{\mrm{ex}}}_M(k)$ is that it will be easier to describe $\mrm{Aut}(\mathcal{E}^{\hat{\mrm{ex}}}_M(k))$, which in turn will allow us to describe $\mrm{Aut}_{\mrm{top}}(\mrm{Aut}(M))$, which is our aim.


	\begin{definition} Given an $L$-structure $M$ we say that $N$ is a definitional expansion of $M$ if $N$ is a structure in a language $L' \supseteq L$ such that $N \restriction L = M$ and, for every symbol in $L' \setminus L$, the interpretation of such symbol in $N$ is $\emptyset$-definable in $M$.
\end{definition}

	\begin{definition}\label{def_expanded_12+}
Let $k \leq \omega$. We define a definitional expansion $\mathcal{E}^{\hat{\mrm{ex}}}_M(k)$ of $\mathcal{E}^{\mrm{ex}}_M(k)$ by adding to $\mathcal{E}^{\mrm{ex}}_M(k)$ the following predicates:
	\begin{enumerate}[(4)]
		\item\label{theP_n_predicates} for $1 \leq n < \omega$, we define an $n+1$-ary predicate $P_n$ on $\mathbf{A}_k(M)$ as follows:
	$$(A, B_1, ..., B_n) \in P^{\mathcal{E}^{\mrm{ex}}_M(k)}_n \; \Leftrightarrow \; A \subseteq \mrm{acl}^{\mrm{g}}_M(B_1 \cup \cdots B_n).$$
\end{enumerate}
\begin{enumerate}[(5)]
	\item\label{def_expanded_12+_item5} a ternary predicate which holds exactly when the following happens:
	$$K_1 \xrightarrow{p_1} K_2 \xrightarrow{p_2} K_3 = K_1 \xrightarrow{p_2 \circ p_1} K_3.$$
\end{enumerate}
\begin{enumerate}[(6)]
	\item a binary predicate which holds exactly when the following happens:
	$$K_1 \xrightarrow{p_1} K_2 \xrightarrow{p_2} K_1 = K_1 \xrightarrow{\mrm{id}_{K_1}} K_1.$$
\end{enumerate}
\begin{enumerate}[(7)]
	\item for every finite group $L$ such that $L \cong \mrm{Aut}(K)$ for some $K \in \mathbf{A}_k(M)$, we add a unary predicate $P_L$ which holds of $B \in \mathbf{A}_k(M)$ if and only if $\mrm{Aut}(B) \cong L$.
\end{enumerate}
When we write $\mathcal{E}^{\hat{\mrm{ex}}}_M$ (so without the $k$) we mean $\mathcal{E}^{\hat{\mrm{ex}}}_M(\omega)$.
\end{definition}

	\begin{remark} If $k = \omega$, then clearly we could replace the predicates $P_n$ above with a single binary predicate which holds of $A, B \in \mathbf{A}(M)$ if and only if $A \subseteq B$.
\end{remark}


The proof of the fact that $\mathcal{E}^{\hat{\mrm{ex}}}_M$ is a definitional expansion of $\mathcal{E}^{\hat{\mrm{ex}}}_M(\omega)$ is standard and we omit the details. We use the added structure \mbox{of $\mathcal{E}^{\hat{\mrm{ex}}}_M(\omega)$ to give a proof of \ref{second_theorem}.}

	\begin{proof}[Proof of Theorem~\ref{second_theorem}] First all, as $\mathcal{E}^{\hat{\mrm{ex}}}_M$ is a definitional expansion of $\mathcal{E}^{\mrm{ex}}_M$ we can use $\mathcal{E}^{\hat{\mrm{ex}}}_M$ and $\mathcal{E}^{\mrm{ex}}_M$ interchangeably. Item (1) was proved in \ref{the_first_crucial_lemma}. We prove (2). To this extent, we show that the assignment $\alpha \mapsto f_{\alpha}$ is a group isomorphism (recall that in (2) we are assuming that $M = N$). Now, the fact that the assignment $\alpha \mapsto f_{\alpha}$ is bijective is easy to see using the assignment $f \mapsto \alpha_f$ from the proof of \ref{the_first_crucial_lemma} and observing that the two maps are one the inverse of the other. We are thus left to show that for $\alpha$ and $\beta$ topological automorphisms of $\mrm{Aut}(M)$ we have that $f_{\beta} \circ f_{\alpha} = f_{\beta \circ \alpha}$, but this is obvious, in fact writing $\alpha(\tilde{p}) \restriction f_\alpha(K)$ the object from the proof of \ref{the_first_crucial_lemma} simply as $f_\alpha(p)$ (and similarly for $\beta$) we have that following:
$$(K, p, K') \xrightarrow{f_{\alpha}} (f_{\alpha}(K), f_{\alpha}(p), f_{\alpha}(K')) \xrightarrow{f_{\beta}} (f_\beta(f_{\alpha}(K)), f_\beta(f_{\alpha}(p)), f_\beta(f_{\alpha}(K'))),$$
and this is the same as:
$$(K, p, K') \xrightarrow{f_{\beta} \circ f_{\alpha}}(f_\beta \circ f_{\alpha}(K), f_\beta \circ f_{\alpha}(p), f_\beta \circ f_{\alpha}(K')),$$
essentially because for any $g \in G$ we obviously have that $\beta \circ \alpha(g) = \beta(\alpha(g))$.
\end{proof}

		The rest of this section is devoted to the application of Theorem~\ref{second_theorem} towards proofs of~\ref{vector_spaces}, \ref{no_alge}, \ref{rationals} and \ref{outer_graphs}. As it will clear from the proofs, the structure added in $\mathcal{E}^{\mrm{ex}}_M(k) \mapsto \mathcal{E}^{\hat{\mrm{ex}}}_M(k)$ is very useful when it comes to describing automorphisms.

	\begin{proof}[Proof of Theorem~\ref{vector_spaces}] Let $V$ be as in the assumptions of the theorem. It is well-known that $V$ is $\omega$-categorical and it has both weak elimination of imaginaries and the small index property (i.e., it has the strong small index property) \cite{evans}. In particular $\mrm{Aut}_\mrm{top}(\mrm{Aut}(V)) = \mrm{Aut}(\mrm{Aut}(V))$ and so, by Theorem~\ref{second_theorem} to understand $\mrm{Aut}(\mrm{GL}(V))$ it suffices to understand $\mrm{Aut}(\mathcal{E}^{\hat{\mrm{ex}}}_V(1)) =: G$, as clearly in this case we have that $k_V = 1$ (where $k_V$ is as in \ref{kM_notation}). Now, every automorphism $f$ of $\mathcal{E}^{\hat{\mrm{ex}}}_M(1)$ induces an automorphism $P(f)$ of the lattice $\mrm{P}(V)$, which we recall is the set of subspaces of $V$ ordered by inclusion, and the correspondence $f \mapsto P(f)$ is a homomorphism. Let $T$ be the kernel of this homomorphism. We claim that:
	\begin{enumerate}[(1)]
	\item $T \cong \mrm{Aut}(F^\times)$;
	\item the  sequence $1 \rightarrow T \rightarrow G \rightarrow \mrm{Aut}(\mrm{P}(V)) \rightarrow 1$ is a split exact sequence.
	\end{enumerate}
We show (1). First of all, recall that the domain of $\mathcal{E}^{\hat{\mrm{ex}}}_V(1)$ is made of triples $(K, p, K')$ where $p: K \cong K'$ and $K$ is a subspace of dimension $0$ or $1$. 
Enumerate the subspaces of dimension $1$ of $V$ as $(K_i : i < \omega)$ and, for every $i < \omega$, choose an element $e_i \in K_i \setminus \{0_V\}$. Then any subspace of dimension $1$ has the form $\{ae_i : a \in F \} =: Fe_i$ for some $i < \omega$, and so any triple $(K, p, K')$ with $p: K \cong K'$ has the form $a e_i \mapsto \lambda_p a e_j$ for some $i, j < \omega$, where $\lambda_p \in F^\times$. Thus, for the rest of the proof we write arbitrary elements of $\mathcal{E}^{\hat{\mrm{ex}}}_V(1)$ as $(Fe_i, \lambda, Fe_j)$. 
We go back to the proof of (1), to this extent, let $f \in T$, then for every $i < \omega$, $f$ sends each triple of the form $(Fe_i, \lambda, Fe_i)$ into a triple of the form $(Fe_i, \lambda', Fe_i)$; so, letting $\lambda' = f_i(\lambda)$, for each $i < \omega$, we have a permutation $f_i$ of $F^\times$. We will show that:
	\begin{enumerate}[(i)]
	\item for every $i < \omega$, $f_i \in \mrm{Aut}(F^\times)$;
	\item for every $i, j < \omega$, $f_i = f_j$.
	\end{enumerate}
Clearly, (i), (ii) implies (1). The proof of (i) and (ii) is standard and we omit it.
%
We are then left to show item (2). To this extent, we first introduce some notation. First of all, by the Fundamental Theorem of Projective Geometry for every $g \in \mrm{Aut}(P(V))$ there is $\hat{g} \in \mrm{Aut}(V)$ which induces $g$. Recall also that we denote with $\alpha \mapsto f_{\alpha}$ the correspondence  that we defined in the proof of \ref{the_first_crucial_lemma}. Now, to define the needed section of $G \rightarrow \mrm{Aut}(P(V))$ that shows short \mbox{exactness we operate as follows:}
	\[ \begin{array}{rcl}	
g \in \mrm{Aut}(P(V))
& \rightsquigarrow&  \hat{g} \in \mrm{Aut}(V)\\
& \rightsquigarrow &  \hat{g} (\cdot) \hat{g}^{-1} =: \alpha_g \in \mrm{Aut}(\mrm{Aut}(V))\\
& \rightsquigarrow &  f_{\alpha_g} \in \mrm{Aut}(\mathcal{E}^{\hat{\mrm{ex}}}_V(1)),
\end{array} \]	
where we denote by $\hat{g} (\cdot) \hat{g}^{-1}$ the inner automorphism induced by $\hat{g} \in \mrm{Aut}(V)$ on $\mrm{Aut}(\mrm{Aut}(V))$.
It should now be clear that the map $g \mapsto f_{\alpha_g}$ is as wanted.
\end{proof}


	\begin{proof}[Proof of \ref{no_alge}] As in \cite[pg. 22]{coarse}, we denote by $\mathcal{E}_M$ the structure with domain $M$ which, for every $0 < n < \omega$, has a relation of arity $2n$ interpreted as the orbit equivalence relation for the action of $\mrm{Aut}(M)$ on $M^n$. If $M$ has no algebraicity, then clearly $k_M = 1$ and $\mathcal{E}_M$ is bi-interpretable \mbox{with $\mathcal{E}^{{\mrm{ex}}}_M(1)$, and so by \ref{second_theorem} we are done.}
\end{proof}

	\begin{proof}[Proof of \ref{free_relational}] This follows from \ref{no_alge} and \cite{ssip_canonical_hom}.
\end{proof}

	The following proposition will be useful in proving Corollaries~\ref{rationals} and \ref{outer_graphs}.

\begin{proposition}\label{useful_prop} Suppose that $M$ is as in \ref{hyp} and that it has no algebraicity. Let $\mrm{Aut}^{\circ}(\mathcal{E}_M)$ be the set of automorphisms of $\mathcal{E}_M$ which preserve the equivalence classes defined by the orbital equivalence relations. Then we have the following:
	\begin{enumerate}[(1)]
	\item $\mrm{Aut}^{\circ}(\mathcal{E}_M) = \mrm{Aut}(M)$;
	\item $\mrm{Aut}^{\circ}(\mathcal{E}_M) \trianglelefteq \mrm{Aut}(\mathcal{E}_M)$;
	\item the bijection $\alpha \mapsto f_{\alpha}$ sends $\mrm{Inn}(\mrm{Aut}(M))$ onto $\mrm{Aut}^{\circ}(\mathcal{E}_M)$ 
	\item $\mrm{Out}(\mrm{Aut}(M)) \cong \mrm{Aut}(\mathcal{E}_M)/\mrm{Aut}^{\circ}(\mathcal{E}_M)$;
	\item given the exact sequences below, the first splits if and only if the second does:
	$$1 \rightarrow \mrm{Inn}(\mrm{Aut}(M)) \rightarrow \mrm{Aut}(\mrm{Aut}(M)) \rightarrow \mrm{Aut}(\mrm{Aut}(M))/\mrm{Inn}(\mrm{Aut}(M)) \rightarrow 1$$
	$$1 \rightarrow \mrm{Aut}^{\circ}(\mathcal{E}_M) \rightarrow \mrm{Aut}(\mathcal{E}_M) \rightarrow \mrm{Aut}(\mathcal{E}_M)/\mrm{Aut}^{\circ}(\mathcal{E}_M) \rightarrow 1.$$
	\end{enumerate}
\end{proposition}

	\begin{proof} This is obvious.
\end{proof}

	\begin{proof}[Proof of \ref{rationals}] It is well-known that $\mathbb{Q}$ has the strong small index property (cf. e.g. \cite[pg. 146]{hodges}). Also, in this case we have no algebraicity and so we can use \ref{useful_prop}. We will use \ref{useful_prop} freely, i.e., without explicitly referring to where exactly it is used. Let $E_n$ be the orbital equivalence relation of arity $2n$. Observe that we have:
	\begin{enumerate}[(a)]
	\item there is a single $E_1$-equivalence class;
	\item there are two $E_2$-equivalence classes of pairs of distinct elements:
	 $$C_1 = \{(a_1, a_2) : a_1 < a_2 \in \mathbb{Q} \} \text{ and } C_2 = \{(a_1,  a_2) : a_1 > a_2  \in \mathbb{Q} \}.$$
\end{enumerate} 
Let $M = \mathbb{Q}$. Every automorphism of $\mathcal{E}_M$ induces a permutation of the set $\{ C_1, C_2\}$. Consider:
$$1 \rightarrow \mrm{Aut}^{\circ}(\mathcal{E}_M) \rightarrow \mrm{Aut}(\mathcal{E}_M) \rightarrow \mrm{Sym}(2) \rightarrow 1.$$
We will show that this sequence is exact and that it splits, this will give us what we want. Thus, we have to show that the map $\mrm{Aut}(\mathcal{E}_M) \rightarrow \mrm{Sym}(2)$ has a section. To see this, we map the unique element of $\mrm{Sym}(2)$ of order $2$ to the permutation $f_s$ of $\mathbb{Q}$ which sends $q$ to $-q$. It is easy to see that the map $f_s$ and is as wanted.
\end{proof}

	\begin{proof}[Proof of \ref{outer_graphs}] It is well-known that all the structures mentioned in the corollary have the strong small index property (cf. e.g. \cite{ssip_canonical_hom} and \cite[pg. 146]{hodges}) and that they satisfy the assumptions of \ref{useful_prop}, so the point is understanding $\mrm{Aut}(\mathcal{E}_M)/\mrm{Aut}^{\circ}(\mathcal{E}_M)$. 

\smallskip \noindent
\underline{The $K_m$-free random graph $R(K_m)$}. Let $E_n$ be the orbital equivalence relation of arity $2n$. Observe that:
	\begin{enumerate}[(a)]
	\item there is a single $E_1$-equivalence class;
	\item there are two $E_2$-equivalence classes of pairs of distinct elements:
	 $$\{(a_1 R a_2) : a_i \in R(K_m) \} \text{ and } \{(a_1 \neg R a_2) : a_i \in R(K_m) \}.$$
\end{enumerate}
Now, a permutation of $R(K_m)$ is an automorphism of $R(K_m)$ if and only if it preserves the two $E_2$-equivalence classes, so suppose that there is a permutation of $R(K_m)$ which switched the two $E_2$-equivalence classes, then it must send the $E_m$-equivalence class of tuples size of size $m$ such that any two pairs are not adjacent in the graph into a clique of size $m$, but by hypothesis the graph $R(K_m)$ is $K_m$-free, a contradiction. Hence, letting $M = R(K_m)$, we have that $\mrm{Aut}(\mathcal{E}_M)/\mrm{Aut}^{\circ}(\mathcal{E}_M)$ is trivial, and so by \ref{useful_prop} we are done.

\smallskip \noindent
\underline{The $m$-colored random graph $R_m$}. For $i \in [1, m]$, let $C_i$ be the binary predicate corresponding to edges of color $i$. Let $E_n$ be the orbital equivalence relation of arity $2n$. Observe that:
	\begin{enumerate}[(a)]
	\item there is a single $E_1$-equivalence class;
	\item there are $m$ $E_2$-equivalence classes of pairs of distinct elements:
	 $$\{(a_1 C_1 a_2) : a_i \in R_m) \}, ..., \{(a_1 C_m a_2) : a_i \in R_m \}.$$
\end{enumerate}
As in previous cases, a permutation of $R_m$ is an automorphism of $R_m$ if and only if it preserves the $E_2$-equivalence classes. On the other hand, for any permutation $\sigma$ of the $m$-many $E_2$-equivalence classes there is a permutation of $R_m$ which induces $\sigma$ and  also preserves the other equivalence relations $E_n$, for $n \geq 3$, in fact taking $R_m$ and changing the color of edges according to $\sigma$ results into a graph isomorphic to $R_m$. Using this and letting $M = R_m$, it is easy to see that $\mrm{Aut}(\mathcal{E}_M)/\mrm{Aut}^{\circ}(\mathcal{E}_M) \cong \mrm{Sym}(n)$, and so by \ref{useful_prop} we are done.
\end{proof}

\end{document}